\documentclass[10pt,reqno]{amsart}

\usepackage{pstricks-add}

\usepackage[pagewise]{lineno}

\usepackage[latin1]{inputenc}
\usepackage{graphicx}
\usepackage{amssymb,amsmath,amstext,amsthm,amsfonts}
\newcommand{\C}{\mathbb{C}}

\theoremstyle{definition}

\newtheorem{lem}{Lemma}
\newtheorem{thm}{Theorem}

\newtheorem{prop}{Proposition}
\newtheorem{rem}{Remark}
\newtheorem{eg}{Example}

\title[]{Surfaces in $\mathbb{R}^7$ obtained from harmonic maps in $S^6$}

\author{Pedro Morais}
\address{Centro de Matem\'{a}tica e Aplica\c{c}\~{o}es (CMA-UBI), Universidade da Beira Interior\\Covilh\~{a}, Portugal}
\email{pmorais@ubi.pt}

\author{Rui Pacheco}
\address{Centro de Matem\'{a}tica e Aplica\c{c}\~{o}es (CMA-UBI), Universidade da Beira Interior\\
 Covilh\~{a}, Portugal}
\email{rpacheco@ubi.pt}

\begin{document}

\maketitle

\begin{abstract}
We will investigate the local geometry of the surfaces in the $7$-dimensional Euclidean space associated to harmonic maps from a Riemann surface $\Sigma$ into $S^6$. By applying methods based on the use of harmonic sequences, we will characterize the conformal harmonic immersions $\varphi:\Sigma\to S^6$ whose associated  immersions  $F:\Sigma\to \mathbb{R}^7$
belong to certain remarkable classes of surfaces, namely: minimal surfaces in hyperspheres; surfaces with parallel mean curvature vector field; pseudo-umbilical surfaces;  isotropic surfaces.
\end{abstract}

\section{Introduction}
It is a well-known fact that any non-conformal harmonic map $\varphi$ from a simply-connected Riemann
surface $\Sigma$ into the round $2$-sphere $S^2$
is the Gauss map of a constant Gauss curvature
surface, $F:\Sigma \to \mathbb{R}^3$, and of two parallel constant mean curvature surfaces, $F^\pm = F\pm\varphi:\Sigma \to \mathbb{R}^3$; the surface $F$ integrates the closed $1$-form $\omega=\varphi\times *d\varphi$, where $\times$ denotes the standard cross product of $\mathbb{R}^3$. The immersion $F$ can also be obtained from  the associated family of $\varphi$ by applying the famous Sym-Bobenko's formula \cite{He,Sym}.

Eschenburg and Quast \cite{EQ} replaced $S^2$
by an arbitrary K\"{a}hler symmetric space $N = G/K$ of compact type
and applied
a natural generalization of Sym-Bobenko's formula \cite{He} to the associated family of a harmonic
map $\varphi:\Sigma \to N$ in order to obtain an immersion $F$ of $\Sigma$ in the Lie algebra $\mathfrak{g}$  of $G$. This construction was subsequently generalized to primitive harmonic maps from $\Sigma$ to generalized flag manifolds \cite{RP}. Surfaces associated to harmonic maps into complex projective spaces have also been constructed in \cite{GSZ,GZ,Za}.

In the present paper, we will investigate the local geometry of surfaces in $\mathbb{R}^7$ associated to harmonic maps from a Riemann surface $\Sigma$ into the nearly K\"{a}hler $6$-sphere $S^6$. In this setting, the harmonicity of a smooth map $\varphi:\Sigma\to S^6$ amounts to the closeness of the differential $1$-form $\omega=\varphi\times *d\varphi$, where $\times$ stands now for the  $7$-dimensional cross product. This means that we can integrate on simply-connected domains in order to obtain a map $F:\Sigma\to \mathbb{R}^7$. When $\varphi$ is a conformal harmonic immersion, $F$ is a conformal immersion; and, in contrast with the $3$-dimensional case, where $F$ is necessarily a totally umbilical surface,  $F$ can exhibit a wide variety of geometrical behaviors in the $7$-dimensional case.
By applying methods based on the use of harmonic sequences \cite{BVW,BW1,BW,DZ,EW, Wolf}, we will characterize the conformal harmonic immersions $\varphi:\Sigma\to S^6$ whose associated  immersions  $F:\Sigma\to \mathbb{R}^7$
belong to certain remarkable classes of surfaces, namely: minimal surfaces in hyperspheres (Theorem \ref{minimal}); surfaces with parallel mean curvature vector field (Theorem \ref{pmcv});  pseudo-umbilical surfaces (Theorem \ref{psnp});  isotropic surfaces (Theorem \ref{isotropic}). From our results, it is interesting to observe that $SO(7)$-congruent harmonic maps into $S^6$ may produce surfaces in different classes of surfaces. For instance,  if $\varphi$ is superconformal in $S^6\cap W$ for some  $4$-plane $W$, then, by Theorem \ref{pmcv}, up to translation, $F$ will be a constant mean curvature surface in $W^\perp$ if $W$ is coassociative; but, if the $4$-plane $W$ admits a \emph{$\times$-compatible decomposition}, then, by Theorem \ref{psnp}, $F$ will be pseudo-umbilical with non-parallel mean curvature vector.

Theorem \ref{minimal} says, in particular, that, if $F$ is minimal in a hypersphere of $\mathbb{R}^7$, then, up to change of orientation of $\Sigma$, $\varphi$ is $SO(7)$-congruent to some almost complex curve. Almost complex curves can be characterized as those weakly conformal harmonic maps from $\Sigma$ in $S^6$ with K\"{a}hler angle $\theta=0$ and were classified in \cite{BVW}. If $\varphi$ is an almost complex curve, $F^+=F+\varphi$ is a constant map and $F^-=F-\varphi$ is just a dilation of $\varphi$.
Apart from almost complex curves, totally real minimal immersions are the most investigated minimal immersions with constant  K\"{a}hler angle  \cite{BVW2}.  In Theorem \ref{toreal} we will identify those totally real harmonic immersions $\varphi:\Sigma \to S^6$ with respect to which both $F^+$ and $F^-$ are immersions with parallel mean curvature vector.


\textbf{Acknowledgements.} This work was  supported by CMA-UBI through the project
UID/MAT/00212/2013.

\section{Preliminaries}

\subsection{The seven dimensional cross product}
Let $\cdot$ be  the standard inner product  on $\mathbb{R}^7$ and $\mathbf{e}_1,\ldots,\mathbf{e}_7$ be the canonical basis of $\mathbb{R}^7$.  Fix the $7$-dimensional cross product $\times$
defined by the multiplication table
\begin{equation}\label{table}
\begin{tabular}{|c|c|c|c|c|c|c|c|}
  \hline
  $\times$ & $\mathbf{e}_1$ &  $\mathbf{e}_2$  &  $\mathbf{e}_3$  &  $\mathbf{e}_4$  &  $\mathbf{e}_5$  &  $\mathbf{e}_6$  &  $\mathbf{e}_7$ \\ \hline
   $\mathbf{e}_1$ & 0 &$\mathbf{e}_3$& $-\mathbf{e}_2$ & $\mathbf{e}_5$ & $-\mathbf{e}_4$ & $-\mathbf{e}_7$& $\mathbf{e}_6$  \\ \hline
   $\mathbf{e}_2$  & $-\mathbf{e}_3$ & 0 & $\mathbf{e}_1$ &$\mathbf{e}_6$ & $\mathbf{e}_7$ & $-\mathbf{e}_4$ & $-\mathbf{e}_5$ \\ \hline
   $\mathbf{e}_3$  & $\mathbf{e}_2$ &$-\mathbf{e}_{1}$ & 0 & $\mathbf{e}_7$ & $-\mathbf{e}_{6}$ & $\mathbf{e}_5$ & $-\mathbf{e}_{4}$\\ \hline
 $\mathbf{e}_4$  &$-\mathbf{e}_5$ &$-\mathbf{e}_6$ & $-\mathbf{e}_7$ & 0 & $\mathbf{e}_1$ & $\mathbf{e}_2$ & $\mathbf{e}_3$ \\ \hline
   $\mathbf{e}_5$  &  $\mathbf{e}_4$ &  $-\mathbf{e}_7$ &  $\mathbf{e}_6$ &   $-\mathbf{e}_1$ & 0 &  $-\mathbf{e}_3$ &  $\mathbf{e}_2$ \\ \hline
   $\mathbf{e}_6$ &  $\mathbf{e}_7$ &  $\mathbf{e}_4$ &  $-\mathbf{e}_5$ &  $-\mathbf{e}_2$ &  $\mathbf{e}_3$ & 0 &  $-\mathbf{e}_1$ \\ \hline
   $\mathbf{e}_7$  &  $-\mathbf{e}_6$ &  $\mathbf{e}_5$&  $\mathbf{e}_4$ &  $-\mathbf{e}_3$ & $-\mathbf{e}_2$ &  $\mathbf{e}_1$ &  $0$ \\
  \hline
\end{tabular}
\end{equation}
The cross product $\times$ satisfies the following identities, for all $\mathbf{x},\mathbf{y}\in\mathbb{R}^7$:
\begin{enumerate}
  \item[(P1)] $\mathbf{x}\cdot (\mathbf{x}\times \mathbf{y})=(\mathbf{x}\times \mathbf{y})\cdot \mathbf{y}=0$;
  \item[(P2)] $(\mathbf{x}\times \mathbf{y})\cdot(\mathbf{x}\times \mathbf{y})=(\mathbf{x}\cdot\mathbf{x})(\mathbf{y}\cdot \mathbf{y})-(\mathbf{x}\cdot \mathbf{y})^2$;
   \item[(P3)] $\mathbf{x}\times \mathbf{y}=-\mathbf{y}\times \mathbf{x}$;
  \item[(P4)] $\mathbf{x}\cdot(\mathbf{y}\times \mathbf{z})=\mathbf{y}\cdot (\mathbf{z}\times \mathbf{x})=\mathbf{z}\cdot(\mathbf{x}\times \mathbf{y})$;
  \item[(P5)] $(\mathbf{x}\times \mathbf{y})\times (\mathbf{x}\times \mathbf{z})=((\mathbf{x}\times \mathbf{y})\times \mathbf{z})\times \mathbf{x}+((\mathbf{y}\times \mathbf{z})\times \mathbf{x})\times \mathbf{x}+ ((\mathbf{z}\times \mathbf{x})\times \mathbf{x})\times \mathbf{y}$;
  \item[(P6)] $\mathbf{x}\times (\mathbf{x}\times \mathbf{y})=-(\mathbf{x}\cdot \mathbf{x})\mathbf{y}+(\mathbf{x}\cdot \mathbf{y})\mathbf{x}$;
  \item[(P7)] $\mathbf{x}\times (\mathbf{y}\times \mathbf{z})+(\mathbf{x}\times \mathbf{y})\times \mathbf{z}=2(\mathbf{x}\cdot \mathbf{z})\mathbf{y}-(\mathbf{x}\cdot \mathbf{y})\mathbf{z}-(\mathbf{y}\cdot \mathbf{z})\mathbf{x}$.
\end{enumerate}

Extend the inner product $\cdot$ and the cross product $\times$  by complex bilinearity to $\mathbb{C}^7=\mathbb{R}^7\otimes {\mathbb{C}}$. We also denote these complex bilinear extensions by $\cdot$ and $\times$, respectively.  Later on we will need the following lemma.
\begin{lem}\label{lemax}
   Let $\mathbf{x}$ and $\mathbf{y}$ be two vectors in $\mathbb{C}^7$  and suppose that $\mathbf{y}\neq 0$ is isotropic. If  $\mathbf{x}\times\mathbf{y} =0$ then $\mathbf{x}$ is isotropic and $\mathbf{x}\cdot \mathbf{y}=0$.
\end{lem}
\begin{proof}
 If $\mathbf{x}$ and $\mathbf{y}$ are collinear, it is clear that $\mathbf{x}$ is also isotropic and $\mathbf{x}\cdot \mathbf{y}=0$. Otherwise,
  if  $\mathbf{x}\times\mathbf{y} =0$, by (P6) we have
  $0=\mathbf{x}\times(\mathbf{x}\times\mathbf{y})=-(\mathbf{x}\cdot \mathbf{x})\mathbf{y}+(\mathbf{x}\cdot \mathbf{y})\mathbf{x},$
  which implies that $\mathbf{x}\cdot \mathbf{x}=0$ (that is, $\mathbf{x}$ is isotropic) and $\mathbf{x}\cdot \mathbf{y}=0$.
\end{proof}

Let  $V\subset \mathbb{R}^7$  be a $3$-dimensional subspace and  set $W:=V^\perp$. The subspace  $V$ is an \emph{associative $3$-plane} if it is closed under the cross product, and
$W$ is a \emph{coassociative $4$-plane} if $V$ is an associative $3$-plane.
The exceptional Lie group $G_2$, which is precisely the group of  isometries  in $SO(7)$ preserving the vector cross product, acts transitively
 on the Grassmannian $G_3^a(\mathbb{R}^7)$ of associative 3-dimensional subspaces of $\mathbb{R}^7$, with isotropy group isomorphic to $SO(4)$ (see \cite{HL} for details).
\begin{lem}\label{vxv}
Let $V$ be a $3$-dimensional subspace of $\mathbb{R}^7$ and set $W:=V^\perp$. Then:
 \begin{enumerate}
   \item[1)] if $V$ is an associative $3$-plane, $V\times W= W$ and $W\times W= V$;
   \item[2)] if $V$ is an associative $3$-plane, for any orthogonal direct sum decomposition
   \begin{equation}\label{W}
      W=W_1\oplus W_2,\,\,\mbox{with $\dim W_1=\dim W_2=2$},
   \end{equation}
   we have $W_1\times W_1=W_2\times W_2$;
\item[3)] conversely, if $W$ admits  an orthogonal direct sum decomposition
\eqref{W}  satisfying $W_1\times W_1=W_2\times W_2$, then $V$ is an  associative $3$-plane;
    \end{enumerate}
\end{lem}
\begin{proof}
  The first two statements can be proved by direct application of the properties of $\times$. Alternatively, one can use the multiplication table \eqref{table} in order to check the statements for a suitable associative $3$-plane, and then apply the transitivity of the $G_2$-action.

  Given a $4$-plane $W$ in the conditions of the third statement, fix an orthonormal basis $\mathbf{v}_1,\mathbf{v}_2,\mathbf{v}_3,\mathbf{v}_4$, with $W_1=\mathrm{span}\{\mathbf{v}_1,\mathbf{v}_2\}$ and
  $W_2=\mathrm{span}\{\mathbf{v}_3,\mathbf{v}_4\}$, such that
  $\mathbf{v}_5:=\mathbf{v}_1\times\mathbf{v}_2=\mathbf{v}_3\times\mathbf{v}_4$.
 By (P1), it is clear that $\mathbf{v}_5\in W^\perp$. Define also $\mathbf{v}_6:=\mathbf{v}_1\times\mathbf{v}_3$ and    $\mathbf{v}_7:=\mathbf{v}_1\times\mathbf{v}_4$.
  One can now apply the properties of $\times$ in order to show that $V=\mathrm{span}\{\mathbf{v}_5,\mathbf{v}_6,\mathbf{v}_7\}$ is an associative $3$-plane.
\end{proof}

Given a $4$-dimensional subspace of $\mathbb{R}^7$, an orthogonal direct sum decomposition of the form \eqref{W} is said to be $\times$-\emph{compatible} if $W_1\times W_1\perp W_2\times W_2$.

\begin{lem}\label{all}
  If the $4$-dimensional subspace $W$ admits one $\times$-compatible decomposition, then any decomposition of $W$ of the form \eqref{W} is $\times$-compatible.
\end{lem}
\begin{proof}
  Fix an orthonormal basis $\mathbf{v}_1,\mathbf{v}_2,\mathbf{v}_3,\mathbf{v}_4$ of $W$ such that the decomposition $W=W_1\oplus W_2$ is $\times$-compatible, where $W_1=\mathrm{span}\{\mathbf{v}_1,\mathbf{v}_2\}$ and $W_2=\mathrm{span}\{\mathbf{v}_3,\mathbf{v}_4\}$, which means that $(\mathbf{v}_1\times \mathbf{v}_2)\cdot(\mathbf{v}_3\times \mathbf{v}_4)=0$. By applying the properties of the cross product, this implies that $(\mathbf{v}_{\sigma(1)}\times \mathbf{v}_{\sigma(2)})\cdot(\mathbf{v}_{\sigma(3)}\times \mathbf{v}_{\sigma(4)})=0$ for any permutation $\sigma$ of $\{1,2,3,4\}$. The result follows now by linearity.
\end{proof}
For example, any decomposition of $W=\mathrm{span}\{\mathbf{e}_1,\mathbf{e}_2,\mathbf{e}_3,\mathbf{e}_4\}$ of the form \eqref{W} is $\times$-compatible. On the contrary, coassociative $4$-planes do not admit $\times$-compatible decompositions.

\subsection{Harmonic sequences} Let $\Sigma$ be a Riemann surface with local conformal coordinate $z$. We will view any smooth map into  the Grassmannian $G_k(\mathbb{C}^n)$ of $k$-planes in $\mathbb{C}^n$ as a vector subbundle of the trivial bundle $\underline{\mathbb{C}}^n=\Sigma\times\mathbb{C}^n$.
 Given a harmonic map $\varphi:\Sigma\rightarrow  G_k(\mathbb{C}^n)$, let $\{\varphi_i\}_{i\in\mathbb{Z}}$, with $\varphi_0=\varphi$, be the correspondent \emph{harmonic sequence}  \cite{BVW,BW1,BW,DZ,EW, Wolf} of $\varphi$.
Any two consecutive elements in this sequence are orthogonal with respect to the standard Hermitian inner product $h(\mathbf{x},\mathbf{y})=\mathbf{x}\cdot \bar{ \mathbf{y}}$ on $\C^{n}=\mathbb{R}^{n}\otimes\mathbb{C}$, where $\cdot$ is the standard inner product of $\mathbb{R}^{n}$.
The harmonic map $\varphi$ has \emph{isotropy order} $r>0$ if $\varphi\perp_h \varphi_i$ for all $0<i\leq r$ but $\varphi\not\perp_h \varphi_{r+1}$.

In the harmonic sequence of $\varphi$, the harmonic maps $\varphi_1$ and $\varphi_{-1}$ are precisely the images of the
\emph{second fundamental forms} $A'_\varphi$ and $A''_\varphi$, respectively. These are defined by $A'_{\varphi}(v)=P_{\varphi}^{\perp_h}(\frac{\partial v}{\partial z})$ and $A''_{\varphi}(v)=P_{\varphi}^{\perp_h}(\frac{\partial v}{\partial \bar z})$, where $P_{\varphi}^{\perp_h}$ is the Hermitian  projection onto the orthogonal bundle of $\varphi$, for any smooth section $v$ of the vector bundle associated to $\varphi$. The harmonicity of $\varphi$ can be reinterpreted in terms of the holomorphicity of  its  second fundamental forms with
respect to the Koszul-Malgrange holomorphic structures on $\varphi$ and on the orthogonal bundle $\varphi^{\perp_h}$.
This allows one to remove singularities and define $\varphi_1$ and $\varphi_{-1}$ globally on $\Sigma$. The remaining elements of the harmonic sequence are obtained by iterating this construction.

Given a harmonic map $\varphi:\Sigma\rightarrow  \mathbb{C}P^n$ into the complex projective space,  any harmonic map $\varphi_i$ in the harmonic sequence of $\varphi$ has the same isotropy order of $\varphi$. Moreover, $\varphi$ is non-constant weakly conformal  if, and only if, it has isotropy order $r\geq 2$.
If $\varphi$ has finite isotropy order $r$, then it is clear that $r\leq n$. Following \cite{BPW},  we say that $\varphi$ is \emph{superconformal} if  $r=n$.
If $\varphi_{j}=\{0\}$  for some $j>0$ (which implies that $\varphi_{-j'}=\{0\}$  for some $j'>0$), then $\varphi$ has infinite isotropy order. Such harmonic maps are  said to be \emph{isotropic} \cite{EW} or \emph{superminimal} \cite{Bryant}. An isotropic harmonic map $\varphi:\Sigma\rightarrow  \mathbb{C}P^n$ has \emph{length} $l>0$ if its  harmonic sequence has length $l$.


Now, consider the totally geodesic immersion of the unit round sphere $S^n$ in $\mathbb{R}P^n\subset \C P^n$. Let $\varphi:\Sigma\to S^n$ be a harmonic map, which can also be seen either as a harmonic map  $\varphi:\Sigma\to  \C P^n$ satisfying $\overline{\varphi}=\varphi$ or
as a parallel section of the corresponding line subbundle. We also have $\overline\varphi_j=\varphi_{-j}$. Moreover,  for each $j$, there exists   a local meromorphic section $f_j$ of $\varphi_j$ such that  \cite{BVW,BW1}:
\begin{align}\label{harmonicsequence1}
  \frac{\partial f_j}{\partial z}=f_{j+1}+\frac{\partial}{\partial z}\log|f_j|^2f_j;\,\,\,
  \frac{\partial f_{j+1}}{\partial \bar z}=-\frac{|f_{j+1}|^2}{|f_j|^2}f_j;\,\,\,
  |f_j||f_{-j}|=1 \,(\mbox{if $f_j\neq 0$}),
\end{align}
with $f_0=\varphi$.  Observe that, since
$|f_1|^2=f_1\cdot \bar f_1= \frac{\partial f_0}{\partial z}\cdot  \frac{\partial f_0}{\partial \bar z}=-\frac{1}{|f_{-1}|^2}f_1\cdot f_{-1},$
we have
\begin{equation}\label{pipi}
  f_1\cdot f_{-1}=-1.
\end{equation}

If $\varphi$ is non-constant weakly conformal, then $\overline{\varphi}_1\perp_h \varphi_1$ ($\varphi$  has isotropy order $r\geq 2$), which means that $f_1\cdot f_1=0$. Differentiating we obtain $f_2\cdot f_1= 0$, that is, $\varphi_2\perp_h \varphi_{-1}$. Hence $\varphi$ has isotropy order $r\geq 3$. This argument can be extended in order to prove the following.
\begin{lem}\label{r=3}
  If the harmonic map  $\varphi:\Sigma \to S^n$ has finite isotropy order $r$, then $r$ is  odd.\end{lem}
  On the other hand,  since  $\varphi_{-j}=\overline{\varphi}_j$,  it is clear that if $\varphi:\Sigma \to S^n$ is an isotropic harmonic map then its length $l$ must also be odd.
\begin{lem}\label{superconformal}
   If $\varphi:\Sigma\to S^n$ has isotropy order $r\geq m$, $m+1$ is even and $\varphi_{k}\neq\{0\}$ is real, with $k=\frac{m+1}{2}$, then the harmonic map $\varphi$ is superconformal in $S^m=S^n\cap W$, where the constant $m+1$-dimensional subspace is given by $W\otimes \C=\varphi_{k}\oplus\bigoplus_{i=-k+1}^{k-1}\varphi_i.$
\end{lem}
\begin{proof}Since $\varphi_k\neq\{0\}$ is real, we have $\varphi_{k+1}=\overline{\varphi}_{k-1}=\varphi_{-k+1}$. Consequently, $\varphi$ has finite isotropy order $r=m$ and
 $W\otimes \C=\varphi_{k}\oplus\bigoplus_{i=-k+1}^{k-1}\varphi_i$ is a constant bundle, that is,
 $\varphi$ is superconformal in  $S^m=S^n\cap W$.
    \end{proof}



\subsection{Almost complex curves in the nearly K\"{a}hler $6$-sphere}
The standard nearly K\"{a}hler structure $J$ on the $6$-dimensional unit sphere $S^6$ can be written in terms of the cross product $\times$ as follows:
$J\mathbf{u}=\mathbf{x}\times \mathbf{u},$
for each $\mathbf{u}\in T_{\mathbf{x}}S^6$ and   $\mathbf{x}\in S^6$.
An \emph{almost complex curve} in  $S^6$ is a non-constant smooth map $\varphi:\Sigma\to S^6$ whose differential is complex linear with respect to the nearly  K\"{a}hler structure $J$. Any such map $\varphi$ is a weakly conformal harmonic map for which $\overline\varphi_1\oplus\varphi\oplus \varphi_1$ is a bundle of associative $3$-planes. In \cite{BVW}, the authors showed that there are four basic types of almost complex curves: (I) linearly full in $S^6$ and superminimal; (II) linearly full in $S^6$ but not superminimal; (III) linearly full in some totally geodesic $S^5$ in $S^6$; (IV) totally geodesic. Twistorial constructions of almost complex curves of type (I) can be found in \cite{BryantO}; almost complex-curves of type (II) and (III) from the $2$-tori  are obtained by integration of commuting Hamiltonian ODE's \cite{BPW,HTU}.

The K\"{a}hler angle $\theta:\Sigma\to [0,\pi]$ of a weakly conformal harmonic map $\varphi:\Sigma \to S^6$ is given by  \cite{BVW}
\begin{equation}\label{kangle}(f_1\times f_{-1})\cdot f_0=-i\cos \theta.\end{equation}
 Almost complex curves are precisely those  weakly conformal harmonic maps with $\theta=0$.
In \cite{BVW}, the authors also proved the following.
\begin{thm}\cite{BVW}\label{bvwp}
Let  $\varphi:\Sigma \to S^6$ be  a weakly conformal harmonic map  with constant K\"{a}hler angle
 $\theta\neq 0,\pi$ and assume that the ellipse of curvature is a circle. Then there is a unit vector $\mathbf{v}$ such that $\varphi$ is linearly full in the totally geodesic $S^5=S^6\cap \mathbf{v}^\perp$. Moreover, there exist an angle $\beta$  and an almost complex curve $\tilde\varphi$ of type (III) such that $\varphi=R_\mathbf{v}(\beta)
 \tilde\varphi$, where $R_\mathbf{v}(\beta)\in SO(7)$  is the rotation defined as follows:
 \begin{equation}\label{rotation}
 R_\mathbf{v}(\beta)\mathbf{x}=\left\{
                                 \begin{array}{ll}
                                   \mathbf{v}, & \hbox{ if $\mathbf{x}=\mathbf{v}$;} \\
                                   \cos \beta\, \mathbf{x}+\sin\beta \,\mathbf{v}\times \mathbf{x}, & \,\,\hbox{if  $\mathbf{x}\in S^6\cap \mathbf{v}^\perp$.}
                                 \end{array}
                               \right.
\end{equation}
 \end{thm}
 These  conditions on $\varphi$ can be expressed in terms of the harmonic sequence as follows.
 \begin{lem}\cite{BVW}\label{KCCE} Let $\varphi:\Sigma \to S^6$ be a  weakly conformal harmonic map. Then, we have:
 \begin{enumerate}
\item[1)]  The ellipse of curvature of $\varphi$ is a circle if, and only if, $\varphi_2$ is an isotropic line bundle.
\item[2)] The ellipse of curvature of $\varphi$ is a point if, and only if, $\varphi_2=\{0\}$.
\item[3)] $\varphi$ has constant K\"{a}hler angle if, and only if, $\varphi_0\times \varphi_1 \perp_h \overline{\varphi}_{2}$.
   \end{enumerate}
 \end{lem}

\section{Surfaces obtained from Harmonic maps into the $6$-sphere}

Let $\Sigma$ be a Riemann surface with local conformal coordinate $z=x+iy$, and let
$\varphi:\Sigma\to S^{n-1}$ be a harmonic map, that is, $\bigtriangleup \varphi \perp T_\varphi S^{n-1}=\{\mathbf{u}\in \mathbb{R}^n|\,\varphi\cdot \mathbf{u}=0\}$.
For $n=7$, this means that $\varphi\times \bigtriangleup \varphi=0$, which is equivalent to the closeness of the one form $\omega=\varphi\times *d\varphi$. Hence, if $\varphi:\Sigma\to S^6$ is a harmonic immersion and $\Sigma$ is simply-connected, we can integrate to
 obtain an immersion $F:\Sigma\to \mathbb{R}^7$ such that $dF=\varphi\times *d\varphi$.
\begin{lem}\label{vv}At corresponding points, the tangent spaces $T\varphi$ and $TF$ satisfy $T\varphi\times T\varphi\perp TF$.
  \end{lem}
\begin{proof}
The tangent space $TF$ at $z=x+iy$ is generated by
$F_x=\varphi\times \varphi_y$ and $F_y=-\varphi\times\varphi_x$; and $T\varphi\times T\varphi$ at $z=x+iy$ is generated by  $\varphi_x\times \varphi_y$.
Taking into account the properties  for the cross product, we have
\begin{align*}
(\varphi_x\times \varphi_y)\cdot (\varphi\times \varphi_y)&=-\varphi_x\cdot\big( \varphi_y\times (\varphi_y\times \varphi)\big)=\varphi_x\cdot(|\varphi_y|^2\varphi-(\varphi\cdot\varphi_y)\varphi_y)\\&=|\varphi_y|^2(\varphi_x\cdot \varphi)-(\varphi_y\cdot\varphi)(\varphi_x\cdot\varphi_y)=0,
\end{align*}
since $\varphi_x\cdot \varphi=\varphi_y\cdot \varphi=0$. Hence $F_x\cdot(\varphi_x\times \varphi_y)=0$; and, similarly,  $F_y\cdot(\varphi_x\times \varphi_y)=0$.
\end{proof}
Making use of the properties for the cross product, we also obtain the following formulae for the first and second fundamental forms of the immersion $F$ in terms of $\varphi$ and its derivatives.
\begin{prop}Let $\mathbf{I}_F$ and $\mathbf{I\!I}_F$ be the first and the second fundamental forms of $F:\Sigma\to \mathbb{R}^7$, respectively. Let $N$ be a vector field of the normal bundle $TF^\perp$.
  Then, with respect to the local conformal coordinate $z=x+iy$ of $\Sigma$, we have
  \begin{equation}\label{fff}
  \mathbf{I}_F=\left(\!\!\!
                   \begin{array}{cc}
                     |\varphi_y|^2 & -\varphi_x\cdot\varphi_y \\
                     -\varphi_x\cdot\varphi_y & |\varphi_x|^2 \\
                   \end{array}
                 \!\!\!\right)\end{equation}
   \begin{equation}\label{sff}
    \mathbf{I\!I}_F^N:=\mathbf{I\!I}_F\cdot N=\left(\!\!\!
                   \begin{array}{cc}
                     (\varphi_x\times \varphi_y)\cdot N +(\varphi\times \varphi_{xy})\cdot N \!\!\!\!\!& (\varphi\times\varphi_{yy})\cdot N \\
                     (\varphi\times\varphi_{yy})\cdot N \!\!\!\!\!& (\varphi_x\times \varphi_y)\cdot N -(\varphi\times \varphi_{xy})\cdot N\\
                   \end{array}
                 \!\!\!\right).\end{equation}
\end{prop}
If  $\varphi$ is conformal, $F$ is also conformal. Let $e^{2\alpha}$ be the common conformal factor of $\varphi$ and $F$. Taking Lemma \ref{vv} into account, one can check that the mean curvature vector of $F$  is given by
$$\mathbf{h}_F:=\frac12\mathrm{tr}\,\mathbf{I}_F^{-1}\mathbf{I\!I}_F=\frac{e^{-2\alpha}}{2}\big\{\mathbf{I\!I}_F(F_x,F_x)+\mathbf{I\!I}_F(F_y,F_y)\big\}=e^{-2\alpha}\varphi_x\times \varphi_y.$$
 Consider the harmonic sequence $\{\varphi_j\}$ of $\varphi:\Sigma\to S^6$, viewed as map
   into $\mathbb{R}P^6\subset \mathbb{C}P^6$. Since $\varphi$ is conformal, we already know that it has isotopy order $r\geq 3$. Let  $\{f_j\}$ be  a sequence of meromorphic  sections
   satisfying \eqref{harmonicsequence1}. We have
   $\varphi=f_0$, $\varphi_z=f_1$, $\varphi_{\bar z}=-\frac{f_{-1}}{|f_{-1}|^2},$ and the conformal factor $e^{2\alpha}$ of $\varphi$ satisfies $\frac{e^{2\alpha}}{2}=|f_{1}|^2=\frac{1}{|f_{-1}|^2}.$ Hence we can also  write   \begin{equation}\label{mean}\mathbf{h}_F=if_1\times f_{-1}.\end{equation}
Observe that $|\mathbf{h}_F|=1$.

In the conformal case,  $\mathbf{u}=e^{-\alpha}(\cos\theta\, F_x + \sin\theta\, F_y)$ is a unit tangent vector for each $\theta\in\mathbb{R}$ and  the ellipse of curvature of $F$ is given by
\begin{equation}\label{huu}
\mathcal{E}_F=\big\{\mathbf{h}_F+e^{-2\alpha}\cos 2\theta\, P^\perp_{TF}(\varphi\times \varphi_{xy})+e^{-2\alpha}\sin 2\theta\, P^\perp_{TF}(\varphi\times \varphi_{xx})|\,\,\, \theta\in \mathbb{R}\big\}.
\end{equation}

Next we will characterize those conformal harmonic maps $\varphi:\Sigma\to S^6$ for which the corresponding immersions  $F:\Sigma\to \mathbb{R}^7$
belong to certain remarkable classes of surfaces, namely: minimal surfaces;  surfaces with parallel mean curvature vector field; pseudo-umbilical surfaces; isotropic surfaces.

\begin{rem}\label{CGC}
When $\varphi$ is a conformal, we have $\mathbf{I}_\varphi=\mathbf{I}_F$. Hence the Gauss curvatures $K_\varphi$ of $\varphi$ and $K_F$ of $F$ coincide on $\Sigma$; and
those conformal harmonic maps $\varphi$ producing a surface $F$ with constant Gauss curvature are precisely those minimal surfaces $\varphi$ of constant curvature in $S^6$. For a general $n$, minimal surfaces of constant Gauss curvature in $S^n$ were classified by R. Bryant \cite{Bryant1}. As a consequence of Bryant's results, if $F$ has constant Gauss curvature $K_F$, then there are three possibilities:
$K_F=1$ and $\varphi$ is totally geodesic; $K_F=1/6$ and $\varphi$ is an open subset of a Boruvka sphere;
$K_F=0$ and $\varphi$ is a flat minimal surface which can be written as a sum of certain exponentials. More precisely, in the flat case, $\varphi$ takes values in $S^5=S^6\cap V$, for some $6$-dimensional subspace $V$ of $\mathbb{R}^7$, and can be written in the form \cite{Bryant1}
\begin{equation}\label{flat}
  \varphi(z)=\sum_{k=1}^3\mathbf{v}_ke^{\mu_k z-\overline{\mu}_k\bar z}+\overline{\mathbf{v}}_ke^{-\mu_kz+\overline{\mu}_k\bar z},
\end{equation}
where $\pm\mu_1,\pm\mu_2,\pm\mu_3$ are six distinct complex numbers in $S^1$ and $\mathbf{v}_k\in V\otimes \mathbb{C}$ are nonzero vectors satisfying
\begin{equation}\label{condflat}\mathbf{v}_k\cdot\mathbf{v}_j=0,\,\,\,\,\,\,\mathbf{v}_k\cdot\overline{\mathbf{v}}_j=0\,\,\mbox{if $j\neq k$},\,\,\,\,\,\, \sum_{k=1}^3 \mathbf{v}_k\cdot\overline{\mathbf{v}}_k=\frac12,\,\,\,\, \sum_{k=1}^3\mu_k^2\,\mathbf{v}_k\cdot\overline{\mathbf{v}}_k=0.\end{equation}
\end{rem}

\subsection{$F$ is minimal in some hypersphere of $\mathbb{R}^7$}
In such cases, the immersions $F$
have parallel mean curvature vector $\mathbf{h}_F$ in the normal bundle of $F$ in $\mathbb{R}^7$ and, simultaneously, are pseudo-umbilical, that is, $\mathbf{I\!I}_F^{\mathbf{h}_F}=\lambda \mathbf{I}_F$ for some smooth function $\lambda$ on $\Sigma$. As a matter of fact, for a general dimension, we have:
\begin{prop}\cite{YC}
  $M^n$ is a pseudo-umbilical submanifold of $\mathbb{R}^m$ such that the mean curvature vector $\mathbf{h}$ is parallel in the normal bundle if, and only if, $M^n$ is either a minimal submanifold of $\mathbb{R}^m$ or a minimal submanifold of a hypersphere of $\mathbb{R}^m$.
\end{prop}


%

\begin{eg}\label{exalm}
Let $\varphi:\Sigma\to S^6$ be an almost complex curve, meaning that $i\varphi_z=\varphi\times \varphi_z$. If $F$ is such that  $dF=\varphi\times *d\varphi$, we have $F_z=i\varphi\times \varphi_z$. Hence $F_z=-\varphi_z$, which means that, up to an additive constant, $F=-\varphi$. In this case, it is clear that $F$ is minimal in $S^6$.
\end{eg}

\begin{thm}\label{minimal}
     Let $\varphi:\Sigma\to S^6$ be a conformal harmonic immersion and $F:\Sigma\to\mathbb{R}^7$  such that $\mathrm{d}F=\varphi\times *\mathrm{d}\varphi$.  $F$ is minimal in some hypersphere of $\mathbb{R}^7$ if, and only if, up to change of orientation of $\Sigma$, one of the following statements holds:
\begin{enumerate}\item[1)] $\varphi$ is an almost complex curve; \item[2)] there exists a unit vector $\mathbf{v}$, an angle $\beta$ and an almost complex curve $\tilde{\varphi}$ of type (III), with $\tilde{\varphi}(\Sigma)\subset S^5=S^6\cap\mathbf{v}^\perp$, such that $\varphi=R_\mathbf{v}(\beta)\tilde\varphi$, where $R_\mathbf{v}(\beta)$ is the rotation given by \eqref{rotation}; \item[3)] $\varphi$ is totally geodesic. \end{enumerate}   
   \end{thm}
\begin{proof}Since $\mathbf{h}_F\neq 0$, $F$ cannot  be minimal in $\mathbb{R}^7$.
We start the proof with the following lemmas.
  \begin{lem}\label{psc}
  $F$ is pseudo-umbilical if, and only if,
  \begin{equation}\label{ps2}\big((f_2\times f_{-1})\times f_{1}\big)\cdot f_0=0.\end{equation}
  \end{lem}
  \begin{proof}
In view of \eqref{fff} and \eqref{sff}, $F$ is pseudo-umbilical if, and only if,
  \begin{equation}\label{lem1}
  (\varphi\times \varphi_{yy})\cdot \mathbf{h}_F=0,\quad(\varphi\times \varphi_{xy})\cdot \mathbf{h}_F=0.\end{equation}
Since   $\mathbf{h}_F$ is a vector field of $TF^\perp$ and $TF\otimes \mathbb{C}=(\varphi\times \varphi_1)\oplus( \varphi\times \varphi_{-1})$, we have
$$(f_0\times f_2)\cdot \mathbf{h}_F=(f_0\times P_{\varphi_1}^\perp {f_0}_{zz})\cdot \mathbf{h}_F=(f_0\times {f_0}_{zz})\cdot \mathbf{h}_F.$$
Equating the real and the imaginary parts, and taking into account that $\triangle \varphi$ is collinear with $\varphi$, we conclude that $(f_0\times f_2)\cdot \mathbf{h}_F=0$ if, and only if, equations \eqref{lem1} hold.

Making use of properties (P3), (P4) and (P7), we see that, since $\mathbf{h}_F$ is given by \eqref{mean},
\begin{align*}(f_0\times f_2)\cdot \mathbf{h}_F=i\big((f_2\times f_{-1})\times f_{1} -2 (f_2\cdot f_1)f_{-1}+(f_2\cdot f_{-1})f_1+(f_{-1}\cdot f_1)f_2 \big)\cdot f_0.
\end{align*}
Now, since $\varphi$ is conformal, by Lemma \ref{r=3} it has isotropy order $r\geq 3$, which means, in particular, that $f_0\cdot f_2=0$.
By definition of harmonic sequence, we also have $f_0\cdot f_1=f_0\cdot f_{-1}=0$.
Hence,  $F$ is pseudo-umbilical if, and only if,
$\big((f_2\times f_{-1})\times f_{1}\big)\cdot f_0=0.$
\end{proof}
\begin{lem}\label{pmc}
  $F$ has parallel mean curvature vector in $TF^\perp$ if, and only if,  $f_2\times f_{-1}=a f_0\times f_{-1}$ for some complex function $a$ on $\Sigma$.
\end{lem}
\begin{proof}
  The mean curvature vector field $\mathbf{h}_F$ is parallel in the normal bundle $TF^\perp$ if, and only if, $\frac{\partial\mathbf{h}_F}{\partial z}$ is a vector field of $TF\otimes \mathbb{C}=(\varphi\times\varphi_1)\oplus(\varphi\times\varphi_{-1})$. Taking the $z$-derivative of  $\mathbf{h}_F=if_1\times f_{-1}$, we obtain, in view of \eqref{harmonicsequence1},
  \begin{equation*}\label{derh}
  \frac{\partial\mathbf{h}_F}{\partial  z}=i(f_2\times f_{-1}+f_1\times f_{0}).
  \end{equation*}
  Hence  $\frac{\partial\mathbf{h}_F}{\partial z}$ is a  vector field of $TF\otimes \mathbb{C}$ if, and only if, $f_2\times f_{-1}= af_0\times f_{-1}+bf_0\times f_{1}$ for some complex functions $a$ and $b$ on $\Sigma$. But
  $$(f_2\times f_{-1})\cdot (f_0\times f_{-1})=-f_0\cdot (f_{-1}\times ( f_{-1}\times f_2))=f_0\cdot\big((f_{-1}\cdot f_{-1})f_2-(f_{-1}\cdot  f_2)f_{-1}  \big)=0;$$
hence the component $b$ of $f_2\times f_{-1}$ along the isotropic section $f_0\times f_{1}$ vanishes.
\end{proof}

\begin{lem}\label{f2f1}
The following statements are equivalent:
 $F$ is minimal in some hypersphere of $\mathbb{R}^7$, that is, $F$ is pseudo-umbilical with parallel mean curvature vector;
$f_2\times f_{-1}=0$. In this case, $\varphi$ has constant K\"{a}hler angle.
\end{lem}
\begin{proof}Assume that $F$ is pseudo-umbilical with parallel mean curvature vector, that is,  condition \eqref{ps2} holds and $f_2\times f_{-1}=a f_0\times f_{-1}$,  for some complex function $a$.
Then, since $$(f_2\times f_{-1})\cdot (f_0\times f_{1})=-\big((f_2\times f_{-1})\times f_{1}\big)\cdot f_0=0,$$  we also have $a=0$. Hence,  $f_2\times f_{-1}=0$.
Conversely, if $f_2\times f_{-1}=0$,  $F$ is pseudo-umbilical by Lemma \ref{psc}, and  $\mathbf{h}_F$ is parallel by Lemma \ref{pmc}.

Suppose that $f_2\times f_{-1}=0$. Taking into account \eqref{harmonicsequence1} and \eqref{pipi},  we have
     \begin{equation}\label{bla}
     (\varphi\cdot \mathbf{h}_F)_z= i\big(f_0\cdot (f_1\times f_{-1})\big)_z=if_0\cdot(f_2\times f_{-1})=0,
     \end{equation}
that is, $\varphi$ has constant K\"{a}hler angle.

\end{proof}

Observe that  the condition $f_2\times f_{-1}=0$ implies, by Lemma  \ref{lemax}, that $f_{2}\cdot f_2=0$, meaning, by Lemma \ref{KCCE}, that the ellipse of curvature  $\mathcal{E}_\varphi$ of $\varphi$ is a point or a circle.

Now, assume that $F$ is minimal in some hypersphere of $\mathbb{R}^7$  and let $\theta$ be its (constant) K\"{a}hler angle. If $\theta=0$, then $\varphi=\pm \mathbf{h}_F$. Consequently,
 $$f_0\times f_1=\pm i (f_1\times f_{-1})\times f_1  = \mp i f_1\times(f_1\times f_{-1})=\mp i(f_1\cdot f_{-1}) f_1=\pm if_1,$$ that is, up to change of orientation, $\varphi$ is an almost complex curve. Next we  assume $\theta\neq 0,\pi$.  We have two possibilities: either
 $\mathcal{E}_\varphi$ is a point or $\mathcal{E}_\varphi$  is a circle.
If $\mathcal{E}_\varphi$ is a point, then $\varphi_2=\{0\}$, which means that $\varphi$ is superminimal in $S^2=S^6\cap \mathbf{E}$, where  $\mathbf{E}\otimes \C:=\varphi\oplus\varphi_{-1}\oplus \varphi_1$ is a (constant) non-associative $3$-plane.   If $\mathcal{E}_\varphi$ is a circle, then, by Theorem \ref{bvwp}, $\varphi=R_\mathbf{v}(\beta)\tilde\varphi$ for some almost complex curve $\tilde{\varphi}$ of type (III) and unit vector $\mathbf{v}$, with $\tilde\varphi(\Sigma)\subset S^6\cap \mathbf{v}^\perp$.

Reciprocally, if $\varphi$ is an almost complex curve, we have, up to translation, $F=-\varphi$, as seen in Example \ref{exalm}, and, consequently, $F$ is minimal in $S^6$. If the conformal harmonic map $\varphi$ is totally geodesic, we have $f_2=0$; then $f_{-1}\times f_2=0$, and, in view of Lemma \ref{f2f1},  $F$ is minimal in some hypersphere. Finally, assume that $\varphi=R_\mathbf{v}(\beta)\tilde\varphi$ for some almost complex curve $\tilde{\varphi}$ of type (III), with $\tilde\varphi(\Sigma)\subset S^6\cap \mathbf{v}^\perp$. Since $R_\mathbf{v}(\beta)\in SO(7)$, we also have $\varphi_i=R_\mathbf{v}(\beta)\tilde \varphi_i$.
 Being an almost complex curve, $\tilde{\varphi}$ certainly satisfies $\tilde \varphi_{-1}\times \tilde \varphi_2=0$.  By applying the properties of $\times$ and the definition of $R_\mathbf{v}(\beta)$, we see now that
$\varphi_{-1}\times \varphi_2=R_\mathbf{v}(\beta)\tilde \varphi_{-1}\times R_\mathbf{v}(\beta)\tilde \varphi_2=0$, which means, by Lemma \ref{f2f1}, that $F$ is minimal in some hypersphere of $\mathbb{R}^7$ .\end{proof}

\begin{eg}\label{trex}
  Consider the conformal harmonic immersion $\varphi:\Sigma\to S^6\cap \mathbf{e}_4^\perp$
 of the form  \eqref{flat} with $\mathbf{v}_k=\frac{1}{2\sqrt{3}} \exp(\pi i/6)(\mathbf{e}_{k}+i\mathbf{e}_{k+4})$ for $k=1,2,3$,
$\mu_1=1$, $\mu_2=\exp(2\pi i/3)$, and $\mu_3=\exp(4\pi i/3)$.
The numbers $\mu_i$ and the vectors $\mathbf{v}_j$ satisfy \eqref{condflat}.
The harmonic map $\varphi$ has constant K\"{a}hler angle $\theta=\pi/2$ and $\varphi_2$ is an isotropic line bundle. Then, by Theorem \ref{bvwp}, $\varphi$ is $SO(7)$-congruent with some almost complex curve. The associated surface $F$ is minimal in a hypersphere of $\mathbb{R}^7$, by Theorem \ref{minimal}, and $F$ has constant Gauss curvature $K_F=0$, by Remark \ref{CGC}. Up to translation,   $F$
is given by
\begin{equation}\label{FVV}F(z)=\sum_{k=1}^3\tilde{\mathbf{v}}_ke^{\mu_k z-\overline{\mu}_k\bar z}+\overline{\tilde{\mathbf{v}}}_ke^{-\mu_kz+\overline{\mu}_k\bar z},\end{equation}
with  $\tilde{\mathbf{v}}_k=\frac{1}{2\sqrt{3}} \exp(2\pi i/3)(\mathbf{e}_{k}+i\mathbf{e}_{k+4})$ for $k=1,2,3$.
\end{eg}

\subsection{$F$ has parallel mean curvature}
\begin{thm}\label{pmcv}
 Let $\varphi:\Sigma\to S^6$ be a conformal harmonic immersion. $F:\Sigma\to\mathbb{R}^7$ has parallel mean curvature vector field and it  is not pseudo-umbilical if, and only if, $\varphi$ is superconformal in some $3$-dimensional sphere $S^3= S^6\cap W$, where $W$ is a coassociative $4$-space.
\end{thm}
\begin{proof}
If  $F$ has parallel mean curvature and it is not pseudo-umbilical, $F$ is not minimal neither in $\mathbb{R}^7$ nor in a hypersphere of $\mathbb{R}^7$.
Since $F$ is not minimal, the normal curvature of $F$ vanishes, by Lemma 4 of \cite{Chen1}. This is known \cite{GR} to be equivalent to the fact that the ellipse of curvature $\mathcal{E}_F$ of $F$ degenerates. Hence, either
$\mathcal{E}_F$ is a point or $\mathcal{E}_F$ is a line segment. But $\mathcal{E}_F$ cannot  be a point because, in that case, in view of \eqref{huu}, we would have $f_0\times f_2=0$, which means that $f_2=0$, and therefore $\varphi$ would be totally geodesic, that is,  $F$ would be minimal,  by Theorem \ref{minimal}. Then  $\mathcal{E}_F$ must be a line segment, which implies, in view of \eqref{huu}, that $\varphi_2$ is real. Consequently,  by Lemma \ref{superconformal},
 $\varphi$
 is superconformal in $S^3=S^6\cap W$, with $W\otimes \C=\varphi\oplus \varphi_{-1}\oplus \varphi_1\oplus \varphi_2$.
We also have, by Lemma \ref{pmc}, $f_0\times f_{-1}=a^{-1} f_2\times f_{-1}$ (with $a\neq 0$); equivalently, $(f_0-a^{-1}{f_2})\times f_{-1}=0$. Hence, by Lemma \ref{lemax}, $f_0-a^{-1}f_2$ is isotropic, which implies that $1+a^{-2}f_2\cdot f_2=0$. Then we  can define orthonormal real vector fields  $\mathbf{v}_1,\mathbf{v}_2,\mathbf{v}_3,\mathbf{v}_4$ by
$$\mathbf{v}_1=f_0,\quad \frac{\mathbf{v}_2+i\mathbf{v}_3}{\sqrt2}=\frac{f_{-1}}{|f_{-1}|},\quad i\mathbf{v}_4=a^{-1}{f_2}.$$
Set $W_1=\mathrm{span}\{\mathbf{v}_1,\mathbf{v}_2\}$ and $W_2=\mathrm{span}\{\mathbf{v}_3,\mathbf{v}_4\}$. Taking the real part of  $f_0\times f_{-1}=a^{-1} f_2\times f_{-1}$, we see that
$\mathbf{v}_1\times\mathbf{v}_2=\mathbf{v}_3\times\mathbf{v}_4$.
Hence, Lemma \ref{vxv} implies that $W^\perp$ is  an associative $3$-plane.

Reciprocally, again by Lemma \ref{vxv}, if $\varphi:\Sigma\to S^3=S^6\cap W$ is a superconformal harmonic map, with  $W^\perp$ an associative
 $3$-plane, $f_0\times f_{-1}$ is complex collinear with $f_2\times f_{-1}$. Consequently, $F:\Sigma\to\mathbb{R}^7$ has parallel mean curvature  and it is clear that $F$ cannot  be pseudo-umbilical (since $\varphi_2$ is real,  $\varphi_2\times \varphi_{-1}\neq 0$).
 \end{proof}
The parallel mean curvature surfaces $F$ arising from a superconformal harmonic map $\varphi$  in some $3$-dimensional sphere $S^3= S^6\cap W$, with $V=W^\perp$ an associative $3$-plane,
are constant mean curvature surfaces in $V$, up to translation, because $F$ satisfies $dF=\varphi\times *d\varphi$ and, by Lemma \ref{vxv}, $W\times W=V$.

\begin{eg}Let $W=\mathrm{span}\{\mathbf{e}_4,\mathbf{e}_5,\mathbf{e}_6,\mathbf{e}_7\}$ and $\varphi:\mathbb{C}\to S^3=S^6\cap W$ be defined by
 $\varphi(x,y)=\frac{1}{\sqrt{2}}\big(\cos x \,\mathbf{e}_4+\sin x\,\mathbf{e}_5+\cos y\,\mathbf{e}_6+\sin y \,\mathbf{e}_7 \big),$
 which is a superconformal harmonic map and parameterizes a Clifford torus. Taking into account the multiplication table \eqref{table}, one can check that the associated surface $F:\mathbb{C}\to V\subset \mathbb{R}^7$ is the cylinder given by
$F(x,y)=\frac{1}{2}\big(-(x+y)\,\mathbf{e}_1-\cos(x-y)\,\mathbf{e}_2+\sin(x-y)\,\mathbf{e}_3\big).$
 \end{eg}

\subsection{$F$ is pseudo-umbilical}
\begin{thm}\label{psnp}
Let $\varphi:\Sigma\to S^6$ be a conformal harmonic immersion. If $F$ is pseudo-umbilical with non-parallel mean curvature vector field, then  $\varphi$ belongs to one of the following classes of harmonic maps:
\begin{enumerate}
 \item[1)] $\varphi$ is full and superminimal in $S^4=S^6\cap W$,   for some   $5$-dimensional vector subspace $W$  of $\mathbb{R}^7$;
  \item[2)]  $\varphi$ has finite isotropy order $r=3$ and is full in $S^4=S^6\cap W$,   for some   $5$-dimensional vector subspace $W$  of $\mathbb{R}^7$;
  \item[3)] $\varphi$  is superconformal in $S^3=S^6\cap W$,  for some $4$-plane $W$ admitting a $\times$-compatible decomposition.  \end{enumerate}
      Conversely, if  $\varphi$  is a superconformal harmonic map in $S^3=S^6\cap W$,  for some $4$-plane $W$ admitting a $\times$-compatible decomposition, then $F$  is pseudo-umbilical with non-parallel mean curvature vector field.
\end{thm}
\begin{proof}
The harmonic map $\varphi$ takes values in $S^6\cap W$ for some $W\subseteq \mathbb{R}^7$ and has isotropy order $r\geq 3$. Assume that $F$ is pseudo-umbilical, that is, $(f_0\times f_2)\cdot \mathbf{h}_F=0$.
 Differentiating this equation and using the properties of cross product, we obtain $(f_0\times f_3)\cdot \mathbf{h}_F=0$. Since $\mathbf{h}_F$ is in the normal
  bundle of $F$ and $f_0\times f_1$ is a section  of $TF\otimes \mathbb{C}$, we also have $(f_0\times f_1)\cdot \mathbf{h}_F=0$. Hence,
  the sections $f_0\times f_i$ and their conjugates $f_0\times f_{-i}$, with $i=1,2,3$, are all orthogonal to $\mathbf{h}_F$.
   This is equivalent to say that $\mathbf{v}=\varphi\times \mathbf{h}_F$ is orthogonal to
the vector subspace $W\otimes \mathbb{C}$, which is  generated by the meromorphic sections $f_0$, $f_i$ and $f_{-i}$, with $i=1,2,3$. Clearly, $\varphi$ is full in $S^6\cap W$.

If $\dim W =7$ (that is, $\varphi$ is full in $S^6$), then  $\mathbf{v}=\varphi\times \mathbf{h}_F$ vanishes, which
implies that $\varphi=\pm \mathbf{h}_F$; and, consequently, as we have seen in the proof of Theorem \ref{minimal},
$\varphi$ is, up to orientation, an almost complex curve, and $F$  has  parallel mean curvature vector.

If $\dim W=6$, then  $\mathbf{v}=f_0\times \mathbf{h}_F$ generates a constant space. Differentiating, we see that $f_0\times (f_2\times f_{-1})+f_0\times (f_1\times f_0)$
and $\mathbf{v}$ are complex linearly dependent. This occurs if, and only if, 
\begin{equation}\label{correc}
f_2\times f_{-1}+f_1\times f_0 =a (f_1\times f_{-1})+b f_0.
\end{equation}
Taking the inner product of both sides of this equation with $f_{-1}$, we obtain $f_0\cdot (f_{1}\times f_{-1})=0$. Then, if we take the inner product of \eqref{correc} with $f_1\times f_{-1}$, we obtain $a=0$. Since $f_2\times f_{-1}+f_1\times f_0$ is isotropic (as a consequence of $F$ be pseudo-umbilical) and $f_0$ is real, then we also have $b=0$, which implies that $f_2\times f_{-1}+f_1\times f_0=0$ -- but this is impossible since $(f_2\times f_{-1})\cdot  (f_0\times f_{-1})=0$.


When $\dim W=3$, $\varphi$ is totally geodesic in $S^6$, and $F$ is minimal in $S^6$, by Theorem \ref{minimal}. So it remains to investigate the cases $\dim W=4$ and $\dim W=5$. In the second case, if $\varphi$ has finite isotropy order $r$, then, by
 Lemma \ref{r=3}, we must have $r=3$; otherwise, $\varphi$ is full and superminimal in $S^4=S^6\cap W$.

Assume now that  $\dim W=4$. In this case, $\varphi_2\neq \{0\}$ is real and $\varphi$  has isotropy
 order $r=3$; hence, by Lemma \ref{superconformal}, $\varphi$ is superconformal in $S^3=S^6\cap W$. Consider the real subspaces $W_1$ and $W_2$ defined by  $W_1\otimes \C=\varphi_{-1}\oplus \varphi_1$ and $W_2\otimes \C=\varphi\oplus \varphi_2$. Since $F$ is pseudo-umbilical, we have  $(f_0\times f_2)\cdot \mathbf{h}_F=0$; consequently, $W_1\times W_1\perp W_2\times W_2$, that is, the decomposition $W=W_1\oplus W_2$ is $\times$-compatible. Conversely, suppose that $\varphi$ is a superconformal harmonic map in $S^3=S^6\cap W$ and that $W=\varphi\oplus \varphi_{-1}\oplus\varphi_1\oplus\varphi_2 $ admits a $\times$-compatible decomposition. Then, by Lemma \ref{all}, the decomposition $W=W_1\oplus W_2$, with $W_1\otimes\C=\varphi_{-1}\oplus \varphi_1$ and $W_2\otimes\C=\varphi\oplus \varphi_2$, is $\times$-compatible at each point of $\Sigma$, which implies that $F$ is pseudo-umbilical: $(f_0\times f_2)\cdot \mathbf{h}_F=0$. Finally, observe that $W$ cannot  be a coassociative $4$-plane because it admits $\times$-compatible decompositions. Hence, by  Theorem \ref{pmcv},  $F$ has non-parallel mean curvature vector field.
\end{proof}

\begin{eg}
Let $W=\mathrm{span}\{\mathbf{e}_1,\mathbf{e}_2,\mathbf{e}_3,\mathbf{e}_4\}$ and $\varphi:\mathbb{C}\to S^3=S^6\cap W$ be the the Clifford torus
$
\varphi(x,y)=\frac{1}{\sqrt{2}}\big(\cos x \,\mathbf{e}_1+\sin x\,\mathbf{e}_2+\cos y\,\mathbf{e}_3+\sin y \,\mathbf{e}_4 \big).$ The corresponding immersion $F:\mathbb{C}\to \mathbb{R}^7$ is given by
$$F(x,y)=\frac12(\cos x\sin y\,\mathbf{e}_1+\sin x\sin y \,\mathbf{e}_2-y\,\mathbf{e_3}+\sin x\cos y\,\mathbf{e}_5-\cos x\cos y\,\mathbf{e}_6+x\,\mathbf{e}_7).$$
Since $W$ admits $\times$-compatible decompositions and $\varphi$ is superconformal in $S^3=S^6\cap W$, the immersion $F$ is pseudo-umbilical with non-parallel mean curvature vector field, by Theorem \ref{psnp}.
\end{eg}

Examples and the existence of pseudo-umbilical surfaces $F$ arising from full harmonic maps $\varphi$  in $S^4=S^6\cap W$, with $W$ a $5$-dimensional subspace of $\mathbb{R}^7$, seem not so easy  to establish.

\subsection{$F$ is isotropic}
An isometric immersion is \emph{isotropic} provided that, at each point, all its normal curvature vectors have the same length \cite{BO}.
\begin{thm}\label{isotropic}Let $\varphi:\Sigma\to S^6$ be a conformal harmonic immersion.
$F$ is isotropic if, and only if, $F$ is pseudo-umbilical and $\varphi_2$ is isotropic (either $\varphi_2=\{0\}$ or $\varphi_2$ is an isotropic line bundle). In particular, if $F$ is isotropic, then, up to change of orientation of $\Sigma$, one of the following statements holds: \begin{enumerate} \item[1)] $\varphi$ is an almost complex curve;\item[2)]  $\varphi=R_\mathbf{v}(\beta)\tilde\varphi$ for some almost complex curve $\tilde{\varphi}$ of type (III) and some rotation $R_\mathbf{v}(\beta)$ of the form \eqref{rotation}; \item[3)] $\varphi$ is totally geodesic;
 \item[4)] $\varphi$ is superminimal in $S^4=S^6\cap W$,   for some   $5$-dimensional vector subspace $W$  of $\mathbb{R}^7$.\end{enumerate}
 \end{thm}
\begin{proof}
  From \eqref{huu} we see that $F$ is isotropic if, and only if, the following holds at each point:
\begin{align}\label{Fumb}
\mathbf{h}_F\cdot P^\perp_{TF}(\varphi\times \varphi_{xy})=&\mathbf{h}_F\cdot P^\perp_{TF}(\varphi\times \varphi_{xx})=0;\\
  |P^\perp_{TF}(\varphi\times \varphi_{xy})|=|P^\perp_{TF}(\varphi\times \varphi_{xx})|;& \,\,  P^\perp_{TF}(\varphi\times \varphi_{xy})\cdot P^\perp_{TF}(\varphi\times \varphi_{xx})=0.\label{Fiso}
\end{align}
Observe that \eqref{Fumb} means that $F$ is pseudo-umbilical. On the other hand,  \eqref{Fiso} means that $P^\perp_{TF}(\varphi\times \varphi_{zz})=\varphi\times P^\perp_{T\varphi}\varphi_{zz}$ is isotropic, that is,
\begin{equation}\label{aux}
0=\big(\varphi\times P^\perp_{T\varphi}\varphi_{zz}\big)\cdot \big(\varphi\times P^\perp_{T\varphi}\varphi_{zz}\big)= P^\perp_{T\varphi}\varphi_{zz}\cdot P^\perp_{T\varphi}\varphi_{zz}-\big(\varphi\cdot P^\perp_{T\varphi}\varphi_{zz}\big)^2.\end{equation}
Since $P^\perp_{T\varphi}\varphi_{zz}$ is a meromorphic section of $\varphi_2$ and $\varphi$ is conformal (in particular, $\varphi\cdot \varphi_2=0$), we see from \eqref{aux} that
\eqref{Fiso} is equivalent to the isotropy of $\varphi_2$.

We know \cite{BVW} that  the ellipse of curvature of an almost complex curve is either a point or a circle, which implies, by Lemma \ref{KCCE}, that either $\varphi_2=\{0\}$ or $\varphi_2$ is an isotropic line bundle. Observe also that if the isotropy order of $\varphi$ is $r=3$, then $\varphi_2$ cannot  be isotropic. Hence, the statement follows now from
Theorem \ref{minimal} and Theorem \ref{psnp}.
\end{proof}

\section{The parallel surfaces}
Given any harmonic map $\varphi:\Sigma\to S^6$, we define also the maps $F^+,F^-:\Sigma\to \mathbb{R}^7$ by $F^\pm=F\pm \varphi$.
We have
\begin{equation}\label{Ipm}
|F^\pm_x|^2=|F^\pm_y|^2=|\varphi_x|^2+|\varphi_y|^2\mp 2\varphi\cdot (\varphi_x\times \varphi_y),\quad F^\pm_x\cdot F^\pm_y=0.
\end{equation}

\subsection{The conformal case}
If $\varphi:\Sigma\to S^6$  is an almost complex curve,  then $F=-\varphi$ up to translation, and  $F^+$ is just a constant map. In such case, $V_\varphi:=\varphi\oplus T\varphi$ is a bundle of associative $3$-planes. Otherwise, we have the following.
\begin{lem}
  Let $\varphi:\Sigma\to S^6$ be a conformal harmonic map such that $V_\varphi$ is everywhere non-associative. Then $F^\pm$ is a conformal immersion with mean curvature vector field given by
\begin{equation}\label{hpm}
  \mathbf{h}^{\pm}=e^{-2\omega^\pm}\big(\varphi_x\times \varphi_y\pm  \frac{\triangle \varphi}{2}\big),
\end{equation}
where  $e^{2\omega^\pm}=2(|\varphi_x|^2\mp \varphi\cdot \big(\varphi_x\times \varphi_y)\big)$ is the conformal factor of $F^\pm$.
\end{lem}
\begin{proof}
  If $\varphi$ is conformal, we have $|\varphi\cdot(\varphi_x\times \varphi_y)|\leq |\varphi_x|^2$, where the equality holds if, and only if, $\varphi$ is collinear with $\varphi_x\times \varphi_y$, that is, if, and only if, $V_\varphi$ is associative at that point. Hence, assuming that  $\varphi$ is everywhere non-associative, we see from \eqref{Ipm} that
  $$|F^\pm_x|^2=|F^\pm_y|^2\geq 2(|\varphi_x|^2- |\varphi\cdot (\varphi_x\times \varphi_y)|)>0,$$
  which means that $F^\pm$ is a conformal immersion. Formula \eqref{hpm} can be deduced by straightforward computations.
\end{proof}
In terms of the meromorphic sections $\{f_j\}$ satisfying \eqref{harmonicsequence1}, the mean curvature vector field $\mathbf{h}^{\pm}$ is given by
\begin{equation}\label{hpmhs}
 \mathbf{h}^{\pm}= \frac{if_1\times f_{-1}\mp f_0}{2\mp 2if_0\cdot(f_1\times f_{-1})},
\end{equation}
assuming that $\varphi$ is conformal.
From this we see that  $|\mathbf{h}^{\pm}|$ is constant if $f_0\cdot(f_1\times f_{-1})$ is constant, that is, if $\varphi$ has constant K\"{a}hler angle. Apart from almost complex curves, {totally real} minimal surfaces are the most investigated  minimal surfaces with constant  K\"{a}hler angle  \cite{BVW2}.  Next we will identify those totally real minimal surfaces $\varphi$ in $S^6$ producing immersions $F^\pm$ with parallel mean curvature.

\begin{thm}\label{toreal}
  Let $\varphi:\Sigma\to S^6$ be a totally real minimal immersion. Then the conformal immersions $F^+=F+ \varphi$ and $F^-=F- \varphi$  have parallel mean curvature vector field if, and only if, up to change of orientation of $\Sigma$, either $\varphi=R_\mathbf{v}(\beta)\tilde\varphi$, for some almost complex curve $\tilde{\varphi}$ of type (III) and some rotation $R_\mathbf{v}(\beta)$ of the form \eqref{rotation}, or $\varphi$ is totally geodesic.
\end{thm}
\begin{proof}Recall that $\varphi$ is totally real if, and only if,
\begin{equation}\label{tr} f_0\cdot(f_1\times f_{-1})=0\end{equation} (the K\"{a}hler angle is $\frac{\pi}{2}$). Consequently, in view of \eqref{hpmhs}, we have $|\mathbf{h}^{\pm}|=\frac{\sqrt{2}}{2}$ and
\begin{equation}\label{hpmtr}
  \mathbf{h}^{\pm}=\frac{i f_1\times f_{-1}\mp f_0}{2}.
\end{equation}
Differentiating \eqref{hpmtr}, we obtain
\begin{equation}\label{deriveh}
 \mathbf{h}^{\pm}_z=\frac{i}{2}f_2\times f_{-1}\mp\big( \frac{f_1}{2}\pm  \frac{i}{2} f_0\times f_1 \big).
\end{equation}
Since the tangent bundle $T^\mathbb{C} F^\pm$ is spanned by the isotropic sections $s_\pm=f_1\pm if_0\times f_1$ and $\bar s_\pm=f_{-1}\mp if_0\times f_{-1}$,   $\mathbf{h}^{\pm}$ is parallel in the normal bundle if, and only if,
$f_2\times f_{-1}=a s_\pm+ b \bar s_\pm$
for some complex functions $a,b$ on $\Sigma$.
From the properties of $\times$ and the properties of the harmonic sequence, we obtain $(f_2\times f_{-1})\cdot  \bar s_\pm=0$. Hence $a=0$ and   \begin{equation}\label{a1}f_2\times f_{-1}=b(f_{-1}\mp if_0\times f_{-1}).\end{equation}

 On the other hand, since $\varphi$ is totally real, the sections
$$f_0,\,\,\, \frac{f_1}{|f_1|},\,\,|f_{1}|f_{-1},\,\,\,f_0\times \frac{f_1}{|f_1|},\,\,\,|f_1| f_0\times  f_{-1},\,\,\,i f_1\times  f_{-1},\,\,\, i f_0\times(f_1\times  f_{-1})$$
form a  moving unitary frame along $\varphi$ (see \cite{BVW2}, pg. 629). Using this, we can write
$$f_2=a_1f_0\times \frac{f_1}{|f_1|}+   a_2 |f_1|f_0\times f_{-1}+a_3 i f_1\times f_{-1}+a_4i f_0\times(f_1\times  f_{-1}),$$
where $a_1,a_2,a_3,a_4$ are complex functions on $\Sigma$.
 By Theorem 4.2 of \cite{BVW2}, these complex functions satisfy $\bar a_3a_4-a_3\bar a_4=0$.
Differentiating \eqref{tr} and applying  property (P4), we see that $f_2\cdot(f_0\times f_{-1})=0$, hence $a_1=0$. We also have
$$ (f_0\times f_{-1})\times f_{-1}=0,\,\,\,  (f_1\times f_{-1})\times f_{-1}=-f_{-1},\,\,\, \big(f_0\times (f_1\times f_{-1})\big)\times  f_{-1}=f_0\times f_{-1}.$$                        Then \begin{equation}\label{a2} f_2\times f_{-1}=-a_3if_{-1}+a_4i f_0\times f_{-1}.\end{equation}

Now, assume that $\mathbf{h}^{\pm}$ is parallel. Equating \eqref{a1} and \eqref{a2},
\begin{equation*}\label{neva}
-a_3if_{-1}+a_4i f_0\times f_{-1}=b(f_{-1}\mp if_0\times f_{-1}),\end{equation*}
we see that $b=\mp a_4=-a_3i$. In particular, we cannot  have $\bar a_3a_4-a_3\bar a_4=0$, unless $a_3=a_4=0$. In this case,
\begin{equation}\label{prem}f_2=a_2 |f_1|f_0\times f_{-1}.\end{equation} Hence $f_2\times f_{-1}=0$, which means, by Lemma \ref{f2f1}, that $F$ is minimal in a hypersphere. Taking Theorem \ref{minimal} into account, we conclude that the totally real minimal immersion $\varphi$ is either of the form $\varphi=R_\mathbf{v}(\beta)\tilde\varphi$, for some almost complex curve $\tilde{\varphi}$ of type (III) and some rotation $R_\mathbf{v}(\beta)$ of the form \eqref{rotation}, or $\varphi$ takes values in $S^2=S^6\cap V$ for some non-associative $3$-space $V$ satisfying $V\times V\subset V^\perp$.

Conversely, assume that the totally real minimal immersion $\varphi$ belong to one of these two classes of conformal immersions. Then, by Theorem \ref{minimal} and Lemma \ref{f2f1}, we certainly have $f_2\times f_{-1}=0$; and, from \eqref{deriveh}, we conclude that both $F^+$ and $F^-$ are immersions with parallel mean curvature vector.
\end{proof}

\begin{rem}\label{rempara}
   Let $\varphi:\Sigma\to S^6$ be a totally real conformal harmonic map such that $F^\pm$ have parallel mean curvature.
   Then $f_0\cdot (f_{-1}\times f_{1})=0$ and, as we have seen in the proof of the previous theorem, we also have $f_{-1}\times f_2=0$. By using these equations, one can check that the parallel surfaces $F^+$ and $F^-$ are given, up to translation, by $F^\pm=if_{-1}\times f_1\pm f_0$ and that both surfaces are minimal in the sphere $S^6(\sqrt{2})$ of $\mathbb{R}^7$ of radius $\sqrt 2$ with center at the origin.

Let $\{F^\pm_i\}$ be the harmonic sequence of the minimal surface $F^\pm$, with $F_0^\pm= F^\pm$.
If $\varphi$ is totally geodesic, we have $\varphi_2=\{0\}$ and $F^\pm_2=\{0\}$. This means that $F^+$ and $F^-$ are totally geodesic. If $\varphi$ is $SO(7)$-congruent with some almost complex curve of type (III), then $\varphi_2$ is an isotropic line bundle and we have $ F^\pm_1=\mathrm{span}\{if_0\times f_1\pm f_1\}$,
   $F^\pm_2=\mathrm{span}\{if_0\times f_2\pm f_2\}$, and $$F^\pm_3=\mathrm{span}\{if_1\times f_2+if_0\times f_3\pm f_3\}.$$
 $ F^\pm_1$ and $ F^\pm_2$ are isotropic line bundles. On the other hand, differentiating \eqref{prem} we see that $f_3$ is collinear with $f_{-1}\times f_1$, and from this it follows, taking the properties of $\times$ into account, that the line bundle  $F^\pm_3$ is not isotropic: $F^\pm_3\cdot F^\pm_3=-2a_2^2|f_1|^2 \neq 0$. Hence $F^\pm$ has finite isotropy order $r=5$.
\end{rem}
\begin{eg}
  The conformal immersion $\varphi:\Sigma\to S^6\cap \mathbf{e}_4^\perp$ of Example \ref{trex} is a totally real minimal immersion and is $SO(7)$-congruent with some almost complex curve. Then  $F^\pm=F\pm\varphi$, with $F$ given by \eqref{FVV}, has parallel mean curvature vector. Moreover, by Remark \ref{rempara}, $F^\pm$ is minimal  in $S^6(\sqrt{2})\cap \mathbf{e}_4^\perp$, with finite isotropy order $r=5$.
 \end{eg}

\subsection{The non-conformal case}
When  $\varphi:\Sigma\to S^2$ is non-conformal, it is well known that  $F^\pm=F\pm\varphi$ are conformal immersions with constant mean curvature $\frac12$.  For non-conformal harmonic maps into $S^6$ we have the following.

\begin{prop}\label{prop1}
    Assume that the harmonic map $\varphi:\Sigma\to S^6$ is everywhere non-conformal and has no branch-points ($d\varphi_p\neq 0$ for all $p\in \Sigma$).
    Then:
    \begin{enumerate}
      \item[1)] $F^+,F^-:\Sigma \to \mathbf{R}^7$ are conformal immersions;
      \item[2)] the mean curvature vectors $\mathbf{h}^{+}$ and $\mathbf{h}^{-}$ of  $F^+$ and $F^-$, respectively, are given by
      \begin{equation}\label{meanvector}\mathbf{h}^{\pm}=e^{-2\omega^\pm}\big(\varphi_x\times \varphi_y\pm  \frac{\triangle \varphi}{2}\big),\end{equation}
      where
      $e^{2\omega^+}$ and  $e^{2\omega^-}$ are the conformal factors of $F^+$ and $F^-$, respectively.
      \item[3)] the conformal immersions $F^+$ and $F^-$ have constant mean curvature along $\varphi$, with  $\mathbf{h}^{\pm}\cdot \varphi=\mp\frac12$.
      \end{enumerate}
 \end{prop}
\begin{proof}
  From \eqref{Ipm}, we see that
 $$|F_x^\pm|^2=|F_y^\pm|^2\geq (|\varphi_x|-|\varphi_y|)^2.$$
If $\varphi$ is everywhere non-conformal and has no branch-points, then   $(|\varphi_x|-|\varphi_y|)^2>0$, and we conclude that $F^+$ and $F^-$ are conformal immersions.

Formula \eqref{meanvector} follows directly from the definitions. Finally,  differentiating twice the equality $\varphi\cdot \varphi=1$ one obtains $\triangle\varphi\cdot  \varphi=-|\varphi_x|^2-|\varphi_y|^2$, and it follows that
  $$ \big(\varphi_x\times \varphi_y\pm  \frac{\triangle \varphi}{2}\big)\cdot \varphi=\mp\frac12\big(|\varphi_x|^2+|\varphi_y|^2\mp 2 \varphi\cdot (\varphi_x\times \varphi_y)\big) =\mp \frac{e^{2\omega_\pm}}{2},$$
  and consequently $\mathbf{h}_{\pm}\cdot \varphi=\mp\frac12$.
\end{proof}

Up to change of orientation, almost complex curves are precisely those weakly conformal harmonic maps $\varphi:\Sigma\to S^6$ satisfying $\varphi_1\times\varphi_{-1}= \varphi$.
We finalize this paper by observing that their non-conformal analogous take values in a $2$-dimensional sphere.
\begin{thm}
  Let $\varphi:\Sigma\to S^6$ be a non-conformal harmonic map such that $\varphi_1\times \varphi_{-1}\subseteq\varphi$.
  Then $\varphi$ takes values in $S^6\cap V$, for some associative $3$-plane $V$. Consequently, under the additional assumptions of Proposition \ref{prop1},
the conformal immersions  $F^+$ and $F^-$ are constant mean curvature immersions in $V$.
\end{thm}
\begin{proof}
If $\varphi_1\times \varphi_{-1}\subseteq \varphi$, we have $f_1\times f_{-1}=a f_0$ for some complex function $a$ on $\Sigma$ and $V_\varphi:=\varphi\oplus\varphi_1\oplus \varphi_{-1}$
is a bundle of associative $3$-spaces. Differentiating, we get
$$f_2\times f_{-1}+f_1\times f_0=a f_1+a_zf_0.$$
Since the fibres of $V_\varphi$ are associative, $ f_0\times f_1$ is a section of $V_\varphi$. Hence   $f_2\times f_{-1}$ is  a section of $V_\varphi$, and,  by associativity,  $f_{-1}\times (f_2\times f_{-1})$ is also a section of $V_\varphi$. But
$$ f_{-1}\times (f_2\times f_{-1})=(f_{-1}\cdot f_{-1})f_2-(f_{-1}\cdot f_2)f_{-1}.$$
This implies that $P_{V_\varphi}^\perp(f_2)=0$, that is, $f_2\in V_\varphi$. Hence $V:=V_\varphi$ is a constant associative $3$-space.
\end{proof}

\end{document}